\documentclass{MBE}
\usepackage{amsmath}
\usepackage{amssymb}
 \usepackage{paralist}
  \usepackage{graphics}
   \usepackage{epsfig}
    \usepackage[colorlinks=true]{hyperref}
 \hypersetup{urlcolor=blue, citecolor=red}

\def\currentvolume{xx}
 \def\currentissue{xx}
  \def\currentyear{20xx}
   \def\currentmonth{xx}
    \def\ppages{1--xx}
     \def\DOI{10.3934/mbe.20xx.xx.xx}

\newfont{\smoldita}{cmmib8}
\newfont{\boldita}{cmmib10}
\newfont{\bboldita}{cmmib10 scaled\magstep1}
\newcommand{\sem}[1]{\mbox{$(e^{t#1})_{t \geq 0}$}}

\newcommand{\mbb}[1]{\mathbb{#1}}
\newcommand{\mb}[1]{\mathbf{#1}}
\newcommand{\ms}[1]{\mathsf{#1}}

\newcommand{\nn}{\nonumber}
\newcommand{\e}{\epsilon}

\newcommand{\p}{\partial}

\newcommand{\cl}[2]{\int\limits_{#1}^{#2}}

\newcommand{\ti}[1]{\tilde{#1}}

\newcommand{\la}{\lambda}

\newcommand{\mc}[1]{\mathcal{#1}}

\newcommand{\comment}[1]{}
\newcommand{\bu}{\boldsymbol\upsilon}
\newcommand{\bof}{\boldsymbol\varphi}

\newtheorem{theorem}{Theorem}[section]
\newtheorem{corollary}{Corollary}

\newtheorem{lemma}[theorem]{Lemma}
\newtheorem{proposition}{Proposition}

\theoremstyle{definition}

\title[A singular limit for a mutation problem]
      {A singular limit for an age structured mutation problem}

\author[J. Banasiak and A. Falkiewicz]{}


 \email{jacek.banasiak@up.ac.za}
 \email{jacek.banasiak@p.lodz.pl}
 \email{aleksandrafalkiewicz@gmail.com}

\thanks{The paper was presented at the conference Micro and Macro Systems in Life Sciences, B\c{e}dlewo, 8-13 June 2015 and was supported by the statutory grant of the Institute of Mathematics of \L\'{o}d\'{z} University of Technology. Participation of A. F. was sponsored by the organizers of the conference.
}

\begin{document}
\maketitle

\centerline{\scshape Jacek Banasiak }
\medskip
{\footnotesize
 \centerline{Department of Mathematics and Applied Mathematics, University of Pretoria, Pretoria, South Africa}
   \centerline{Institute of Mathematics,
Technical University of \L\'{o}d\'{z}, \L\'{o}d\'{z}, Poland}
} 

\medskip

\centerline{\scshape Aleksandra Falkiewicz}
\medskip
{\footnotesize
\centerline{Institute of Mathematics,
Technical University of \L\'{o}d\'{z}, \L\'{o}d\'{z}, Poland}   \centerline{}
} %

\bigskip

\centerline{(Communicated by the associate editor name)}

\begin{abstract}
The spread of a particular trait in a cell population often is modelled by an appropriate system of ordinary differential equations describing how the sizes of subpopulations of the cells with the same genome change in time. On the other hand, it is recognized that cells have their own vital dynamics and mutations, leading to changes in their genome, mostly occurring during the cell division at the end of its life cycle. In this context, the  process is described by a system of McKendrick type equations  which resembles a network transport problem. In this paper we show that, under an appropriate scaling of the latter,   these two descriptions are asymptotically equivalent.
\end{abstract}

\section{Introduction}

An important problem related to mutations is to understand how a particular trait spreads in a population. One of the simplest ways to model this is to  describe the change in the sizes of subpopulations having this trait. It can be done by the standard balancing argument: the rate of change of the number of, say, cells with a particular genome $\gamma$ is equal to the rate of recruitment of cells with other genomes that change, due to mutations, to $\gamma$, minus the rate at which the mutations cause the cells with genome $\gamma$ to move to subpopulations with another genome. This balance equation typically is supplemented by terms describing the death and proliferation of the cells. Models of this type were extensively studied, see for instance  \cite{BM, BoK, KS, AD+P1, SPK}, either in the context of the development of drug resistance in cancer cells, or in describing the micro-satellite repeats. In the former, a population of cells was divided into smaller subpopulations according to the number of the drug resistant gene,  while in the latter the feature of interest was the number of the micro-satellite repeats. In both cases the authors only allowed for changes between neighbouring populations, which resulted in the birth-and-death system with proliferation,
\begin{eqnarray}
u_0' &=& a_0u_0 + d_1 u_1,\nn\\
 u_1'&=& a_1u_1 + d_2u_2,\nn\\
  u_n' &=& a_nu_n + b_{n-1} u_{n-1} + d_{n+1}u_{n+1}, \quad n \geq 2,\label{1}
\end{eqnarray}
where $u_n$ is the number of cells belonging to the subpopulation  $n$, $n=0,1,2\ldots$ (e.g. having $n$ drug-resistant genes), $d_{n+1}, b_{n-1}$ are the rates of recruitment from the populations $n+1$ and $n-1$ into the population $n$ and $a_n$ is the net growth rate of the population $n$ which  incorporates birth, death and loss to other populations of cells of type $n$.  We note that the particular form of (\ref{1}), in which the first equation is decoupled from the rest, is due to the assumption adopted in {\em op.cit.} that any object in the state $0$ can generate only objects in the same state.  This assumption, however, is of no significance in our considerations.

It is clear that, in principle,  jumps between arbitrary populations can occur and thus there is no need to restrict our attention to tridiagonal matrices. We can consider a general model
\begin{equation}
\mb u' =  \mb L\mb u,
\label{B}
\end{equation}
where $\mb u = (u_i)_{i\in \ms N}$ and $\mb  L = (l_{ij})_{i,j \in \ms N}$, $\ms N\subseteq \mbb N$ may be infinite and $\mb  L$ is a positive off-diagonal matrix, where the off-diagonal term $l_{ij}$ is the rate at which the cells are recruited into the population $i,$ characterized by a particular genotype,  from the population $j$. The diagonal entries represent the loss rates from corresponding classes.

At the same time it is recognized that the cells have their own vital dynamics that should be taken into account if a more detailed model of the evolution of the whole population is to be built. Also, the mutations can be divided into various groups. Here, we distinguish two types of mutations: those that are due to the replication errors and occur when the cell divides, and others, due to external factors (mutagenes), that may happen at any moment of the cell's life cycle. This results in a model of the form
\begin{eqnarray}
\p_t \mb u(x,t) + \mb  V\p_x\mb u(x,t) &=&-\mb  M \mb  u(x,t)+ \mb  R \mb u(x,t),\quad x\in(0,1), t\geq 0,\nn\\
\mb u(x,0) &=& \mathring{\mb u}(x),\nn\\
\mb u(0,t) &=& \mb  K\mb u(1,t),
\label{rot1}
\end{eqnarray}
where we consider a population of cells described by their density $u_j(x)$ in each class $j \in \ms N,$ and where $x$ is a parameter describing the maturity of the cell, typically its age or size. In general, cells in each class can divide at a different age (or size), say $\tau_j,$ so that we normalize the age of division to $1$ by introducing the maturation velocities $\mb  V = \mathrm{diag}(v_i)_{i\in \ms N}$, where $v_i=1/\tau_i$. Cells are born at $x = 0$ and divide at $x = 1$ producing, due to mutations during mitosis, daughter cells of an arbitrary class, the distribution of which is  governed by a nonnegative matrix $\mb  K = (k_{ij})_{i,j\in \ms N}.$ Though typical mitosis produces two daughter particles, this is not always the case. For instance,  cancer cells can produce up to five daughter cells, \cite{TWC}, so that we shall not place any further restrictions on $\mb K$. Further,  $\mb  M= \mathrm{diag}(\mu_j)_{j\in \ms  N}$ gives the death rates due to external causes and $\mb {R}=(r_{ij})_{i,j\in \ms N}$ describes redistribution of cells caused by mutations due to external factors (mutagenes). Finally, the nonnegative vector $\mathring{\mb u}$ describes the initial population. A classical example of this type is the discrete  Lebowitz-Rubinow-Rotenberg model in which the cells' populations are distinguished precisely by the maturation velocities, \cite{rot}.

The main objective of this paper is to determine under what conditions can solutions to (\ref{rot1}) be approximated by the solutions of (\ref{B}). There could be several ways to approach this problem and the answer may be not unique. Our approach is to assume that the maturation velocities are very large or, in other words, the cells divide many times in the reference unit of time, while the deaths and mutations due to external causes remain at fixed, independent of the maturation velocity, levels. To balance the fact that there is a large number of cell divisions in the unit time, we assume that the daughter cells have a tendency to be of the same genotype as the mother (see e.g. \cite[p. 19]{leb}), which is represented by splitting the boundary operator as $\mb K= \mb I+\e\mb B$, $\e\ll 1$. Thus, we will consider the singularly perturbed problem \begin{eqnarray}
\p_t \mb u_\e(x,t) + \e^{-1}\mb  V\p_x\mb u_\e(x,t) &=&-\mb  M \mb  u_\e(x,t)+ \mb  R \mb u_\e(x,t),\quad x\in(0,1), t\geq 0,\nn\\
\mb u_\e(x,0) &=& \mathring{\mb u}(x),\nn\\
\mb u_\e(0,t) &=& (\mb  I +\e \mb  B)\mb u_\e(1,t),
\label{rot1e}
\end{eqnarray}
where $\mb  I$ is the identity matrix. By $l^1_{\ms N}$ we denote $\mbb R^\ms N$ equipped with the $l^1$ norm if $\ms N$ is finite or $l_1$ (the space of absolutely summable sequences) if $\ms N$ is infinite (thus in the latter case we identify $\ms N$ with $\mbb N$). If $\ms N =\mbb N$, we assume that all matrices in (\ref{rot1e}) represent bounded operators from $l^1_\ms N$ to $l^1_\ms N$.     Alongside (\ref{rot1e}), we  consider the simplified problem
\begin{eqnarray}
\p_t \mb u_\e(x,t) + \e^{-1}\mb  V\p_x\mb u_\e(x,t) &=&0,\quad x\in(0,1), t\geq 0,\nn\\
\mb u_\e(x,0) &=& \mathring{\mb u}(x),\nn\\
\mb u_\e(0,t) &=& (\mb  I +\e \mb  B)\mb u_\e(1,t).
\label{rot2e}
\end{eqnarray}
 We will be working in $\mb X = L_1([0,1], l^1_{\ms N}).$  We define $\mb A_\e$ as the realization of the differential expression  ${\sf A}_\e = \mathrm{diag}(-\e^{-1}{v_j} \p_x)_{j\in \ms N}-\mb  M+\mb  R$ on the domain
$D(\mb A_\e) = \{\mb u \in \mb W^{1}_1([0,1],l^1_{\ms N});\; \mb u(0) = (\mb  I+\e \mb  B)\mb u(1)\}$, see \cite{BD08}. In line with the interpretation of the model, we assume that $\mb  I +\e \mb  B\geq 0$. Further,  $\mb A_{0,\e}$ is the realization of  ${\sf A}_{0,\e} = \mathrm{diag}\{-\e^{-1}{v_j}\p_x\}_{j\in \ms N}$ on the same domain. We assume that if $\ms N$ is infinite, then there are numbers $v_{\min}$ and $v_{\max}$ such that
\begin{equation}
0<v_{\min}\leq v_j\leq v_{\max}<+\infty,\quad j\in \ms N.
\label{cmax}
\end{equation}
If $\ms N$ is finite, the fact that for each $\e>0$ the operator $(\mb A_{0,\e}, D(\mb A_\e))$ generates a semigroup, denoted by $\sem{\mb A_{0,\e}}$, follows from  \cite[Theorem 3.1]{BFN2}. However, the argument used in \textit{op.cit} is purely norm dependent and can be repeated for countable $\ms N$ as long as $\mb B$ is bounded and (\ref{cmax}) is satisfied, see also \cite{BD08}. Then the boundedness of $\mb M$ and $\mb R$  allows for the application of the  Bounded Perturbation Theorem, \cite[Theorem III 1.3]{EN}, which gives the generation of $\sem{\mb A_{\e}}$ by $(\mb A_{\e}, D(\mb A_\e)).$

The paper is organized as follows. In Section \ref{s2} we show the convergence of the resolvents of (\ref{rot1e}) as $\e\to 0^+$ and thus, by the Trotter-Kato theorem, the convergence of semigroups \sem{\mb A_{\e}} to the age independent semigroup generated by $\mb V\mb B-\mb M+\mb R,$ but only for initial values that are constant for each $j\in \ms N$. Such a convergence is often referred to as a regular convergence of semigroups, \cite{Bobks1}. This result is not fully satisfactory as often singularly perturbed  semigroups, possibly corrected by initial or boundary layers, converge to the limit semigroup for all initial conditions, see \cite{BaLabook, BFN3}. In Section \ref{sec2} we show that in this case, in general it is impossible to achieve such a result. However, if we restrict our attention to the macroscopic characteristic of the model; that is, the observable total size of each subpopulation $j$, obtained from the solution to (\ref{rot2e}), then the situation changes. In Section \ref{sec3} we prove that if the maturation times in each subpopulation are natural multiples of some fixed reference time, then, for arbitrary initial values $\mathring{\mb u}$, the population sizes obtained from the solution to (\ref{rot2e}) by integration over $(0,1)$, converge to the solution to (\ref{B}) that emanates from the initial condition $\int_0^1\mathring{\mb u}(x)dx$. This extends a similar result from \cite{BFN3} obtained for velocities independent of $j\in \ms N$.

\section{Regular convergence}\label{s2}

\begin{lemma} The operators $(\mb A_{\e}, D(\mb A_\e))$ generate equibounded  $C_0$-semigroups on $\mb X$.
\end{lemma}
\begin{proof} We begin with $(\mb A_{0,\e}, D(\mb A_\e)).$ As we mentioned earlier, that $(\mb A_{0,\e}, D(\mb A_\e))$ generates a $C_0$ semigroup for each $\e>0$ follows from \cite[Theorem 3.1]{BFN2}. Clearly, if for each $\e>0$ the semigroup \sem{\mb A_{0,\e}} is contractive, then there is nothing to prove.

  In the general case, we rewrite the main formulae from the proof of \cite[Theorem 3.1]{BFN2}, specified for (\ref{rot2e}). First, we solve
\begin{equation}
\la \e u_{\e,j} + v_j \p_x u_{\e,j} =\e f_j, \quad j\in \ms N,\quad x \in (0,1),
\label{res0}
\end{equation}
with $(u_{\e,j})_{j\in \ms N}=\mb u_\e \in D(\mb A_{0,\e})$. The general solution is
\begin{equation}
\mb u_\e(x) = \mb  E_{\e\la}(x)\mb c_\e +\e\mb  V^{-1} \cl{0}{x}\mb  E_{\e\la}(x-s)\mb f(s)ds,
\label{res1}
\end{equation}
where $\mb c_\e = (c_{\e,j})_{j\in \ms N}$ is an arbitrary vector and $\mb  E_{\e\la}(s) = \mathrm{diag}\left(e^{-\frac{\e\la}{v_j}s}\right)_{j\in \ms N}.$  Then, using the boundary condition $\mb u_\e(0) = (\mb  I+\e\mb  B)\mb u_\e(1),$ we obtain
\begin{equation}
(\mb  I - (\mb  I+\e\mb  B)\mb  E_{\e\la}(1))\mb c_\e = \e(\mb  I+\e\mb  B)\mb  V^{-1} \cl{0}{1}\mb  E_{\e\la}(1-s)\mb f(s)ds.
\label{resk}
\end{equation}
If $\e$ is fixed,  $\|\mb E_{\e\la}(1)\|$ can be made arbitrarily small by choosing large $\la$. Then $\mb c_\e$ is uniquely defined by the Neumann series and hence the resolvent of $\mb A_{0,\e}$ exists.

Since   $\mb  I+\e\mb  B\geq 0, $ we can work with $\mb f\geq 0$. Adding together the rows in (\ref{resk})
we obtain
\begin{equation}
\sum\limits_{j\in \ms N}  c_{\e,j} = \sum\limits_{j\in \ms N} (1+\e\ms b_j)e^{-\frac{\e\la}{v_j}}c_{\e,j} + \e\sum\limits_{j\in \ms N} \frac{1+\e\ms b_j}{v_j} \cl{0}{1}e^{\frac{\e\la}{v_j}(s-1)}f_j(s)ds,
\label{res2}
\end{equation}
where $\ms b_j = \sum\limits_{i\in N} b_{ij}<+\infty$ for any $j \in \ms N$ (by assumption that $\mb  B$ is a bounded operator on $l^1_\ms N$).
We renorm  $\mb X$ with  $\|\mb u\|_v = \sum_{j\in \ms N} v^{-1}_j\|u_j\|$, which is equivalent to the standard norm by (\ref{cmax}). Then, by integrating (\ref{res1}), we obtain
\begin{equation}\|\mb u_\e\|_v = \label{vest} \frac{1}{\la} \sum\limits_{j\in \ms N}\! c_{\e,j}e^{-\frac{\e\la}{v_j}}\ms b_j + \frac{\e}{\la} \sum\limits_{j\in \ms N}\frac{\ms b_j}{v_j}\!\! \cl{0}{1} \!\! e^{\frac{\e\la}{v_j}(s-1)}f_j(s)ds + \frac{1}{\la} \|\mb f\|_v, 
\end{equation}
 see \cite[Theorem 3.1]{BFN2} for details. The case when $\ms b_j\leq 0$ for all $j\in \ms N$ leads to a contractive semigroup for each $\e>0$ and the equiboundedness is obvious. Otherwise, as in \textit{op. cit.}, we can assume that each $\ms b_j$ is nonnegative. Then the generation is achieved by the application of the Arendt-Batty-Robinson theorem \cite{Ar, BaAr}. However, to prove equiboundedness of \sem{\mb A_{0,\e}}  we need to understand how the Hille-Yosida estimates are obtained in this theorem.

Using, for instance, the proof given in \cite[Theorem 3.39]{BaAr}, a densely defined resolvent positive operator $T$ on a Banach lattice $X$ generates a positive semigroup if there are $\la_0>s(T)$ (the spectral bound of $T$) and $c>0$ such that for any nonnegative $x\in X$ we  have $
\|R(\la_0,T)x\|\geq c\|x\|.
$
In the proof one defines $S = T-\omega I$, where $s(T)<\omega\leq \la_0$ and then the Hille-Yosida estimate for $S$ is obtained as
\begin{equation}
\|\la^n R(\la, S)^n x\| \leq c^{-1}\|R(0,S)\|\|x\|= c^{-1}\|R(\omega, T)\|\|x\|, \quad \la>0.
\label{abr}
\end{equation}
Hence,  we have to show that $s:=\sup\limits_{\e>0}\{s(\mb A_{0,\e})\}<\infty$ (and thus, by (\ref{vest}), we can take $c=\la^{-1}$ for any fixed $\la>s$) and that for some $\omega>s$ the family $\{\|R(\omega, \mb A_{0,\e})\|\}_{\e>0}$ is equibounded.
For this, let us return to (\ref{resk}) and estimate the Neumann series for $(\mb  I - (\mb  I+\e\mb  B)\mb  E_{\e\la}(1))^{-1}$. We have
\begin{eqnarray}
\|(\mb  I - (\mb  I+\e\mb  B)\mb  E_{\e\la}(1))^{-1}\| &\leq& \sum\limits_{n=0}^\infty (1+\e\|\mb  B\|)^n  e^{-\frac{\e\la n}{v_{\max}}}\leq  \sum\limits_{n=0}^\infty e^{\e \|\mb  B\|n } e^{-\frac{\e\la n}{v_{\max}}}\nn\\&=& \frac{1}{1-e^{\e(\|\mb  B\| - v_{\max}^{-1}\la)}},\label{neuest}
\end{eqnarray}
provided $\la > \|\mb  B\|v_{\max}$.

Next, using l'H\^{o}spital's rule, we find
$$
\lim\limits_{\e\to 0^+} \frac{\e}{1-e^{\e(\|\mb  B\| - v_{\max}^{-1}\la)}} = \frac{1}{v_{\max}^{-1}\la-\|\mb  B\|}
$$
and hence
$
\e \|(\mb  I - (\mb  I+\e\mb  B)\mb E_{\e\la}(1))^{-1}\| \leq L_1
$ for some constant $L_1$ (depending on $\la$), yielding
\begin{equation}
\|\mb c_\e\| = \e \left\|(\mb  I - (\mb  I+\e\mb  B)\mb E_{\e\la}(1))^{-1}(\mb  I+\e\mb  B)\mb  V^{-1} \cl{0}{1}\mb  E_{\e\la}(1-s)\mb f(s)ds\right\|\leq L_2\|\mb f\|
\label{cep}
\end{equation}
for some constant $L_2$. Let us fix  $\omega > v_{\max}\|\mb  B\|$. Using (\ref{res1}) and (\ref{cep}), we see that there is a constant $L$, independent of $\e$  such that \begin{equation}
\|R(\omega, \mb A_{0,\e})\|\leq {L}.
\label{abr0}
\end{equation}
 Then using (\ref{vest}), the equivalence of the norms $\|\cdot\|_v$ and $\|\cdot\|$ and  $R(\la, \mb A_{0,\e}-\omega \mb I) = R(\la +\omega, \mb A_{0,\e})$, we write (\ref{abr}) as
\begin{equation}
\|R(\la, \mb A_{0,\e})^n\| \leq \frac{M}{(\la-\omega)^n}
\label{abr1}
\end{equation}
for some constant $M$ and $\la >\omega$. Thus
the operators $(\mb A_{0,\e}, D(\mb A_\e))$ generate semigroups satisfying
$$
\|e^{t\mb A_{0,\e}}\| \leq Me^{\omega t}
$$
with constants $M$ and $\omega$ independent of $\e$.

The result for $\mb A_\e$ follows from \cite[Theorem III.1.3]{EN}.

\end{proof}

Now let us pass the the question of convergence of the resolvents. We introduce the projection
 operator $\mb  P: \mb X \to l^1_{\ms N}$ by
\begin{equation}
\mb  P f = \cl{0}{1}\mb f(s)ds = (\cl{0}{1}f_1(s)ds, \ldots, \cl{0}{1}f_n(s)ds,\ldots).
\label{mbP}
\end{equation}
\begin{theorem}
If $\la > \omega + \|\mb  Q\|L\omega^{-1}$, where $\omega$ and $L$ are defined in (\ref{abr0}) and (\ref{abr1}), and $\mb  Q = -\mb  M+\mb  R$, then
\begin{equation}
\lim\limits_{\e\to 0^+}R(\la, \mb A_\e) =  R(\la, \mb  V\mb  B  +\mb  Q)\mb  P.
\label{limae}
\end{equation}
in the uniform operator topology. \label{rescon}
\end{theorem}
\begin{proof}
By the previous proof, the resolvent of $\mb A_{0,\e}$ is given by
\begin{equation}
[R(\la,\mb A_{0,\e})\mb f](x) = \mb  E_{\e\la}(x)\mb c_\e +\e\mb  V^{-1} \cl{0}{x}\mb  E_{\e\la}(x-s)\mb f(s)ds,\qquad \la >v_{\max}\|\mb  B\|,
\label{res1a}
\end{equation}
where
\begin{equation}
\mb c_\e = \e(\mb  I - (\mb  I+\e\mb  B)\mb  E_{\e\la}(1))^{-1}(\mb  I+\e\mb  B)\mb  V^{-1} \cl{0}{1}\mb  E_{\e\la}(1-s)\mb f(s)ds.
\label{resk1}
\end{equation}
We observe that for any $\alpha\in [0,1]$ the matrix $\mb  E_{\e\la}(\alpha)$ has the following Taylor expansion
\begin{equation}
\mb  E_{\e\la}(\alpha) = \mb  I+\e\mb  R_0(\alpha) = \mb  I -\e\la \alpha\mb  V^{-1} + \e^2\mb  R_1(\alpha),
\label{E}
\end{equation}
where, using the integral form of the reminders, we find
$$
\|\mb  R_0(\alpha)\|\leq \frac{\la\alpha}{v_{\max}}\qquad  \|\mb  R_1(\alpha)\|\leq\frac{\la^2\alpha^2}{2v^2_{\max}}.
$$
First we see that
\begin{eqnarray*}
\left\|\e\mb  V^{-1} \cl{0}{x}\mb  E_{\e\la}(x-s)\mb f(s)ds\right\| &\leq& \e\|\mb  V^{-1}\| \cl{0}{1} \left\| \cl{0}{x}\left(e^{-\frac{\la\e(x-s)}{v_j}} f_j(s)\right)_{j\in \sf N}ds\right\|dx \\
 \leq \e\|\mb  V^{-1}\| \cl{0}{1}\cl{0}{x}\| \mb f(s)\|ds dx&\leq& \e\|\mb  V^{-1}\| \|\mb f\|
\end{eqnarray*}
and hence the last term in (\ref{res1a}) converges to zero as $\e\to 0^+$.

Next, using the second equality in (\ref{E}) with $\alpha =1$, we have for $\la > v_{\max}\|\mb  B\|$
\begin{eqnarray*}
\e(\mb  I - (\mb  I+\e\mb  B)\mb  E_{\e\la}(1))^{-1} &=& \e(\mb  I - (\mb  I+\e\mb  B)(\mb  I -\e\la \mb  V^{-1} + \e^2\mb  R_1))^{-1} \\
&=& (\la \mb  V^{-1}-\mb  B -\e\mb  R_1 +\e\la \mb  B\mb  V^{-1} - \e^2\mb  B\mb  R_1)^{-1}.
\end{eqnarray*}
Since $\la \mb  V^{-1}-\mb  B = \mb  V^{-1}(\la \mb  I - \mb  V\mb  B)$ is invertible for $\la >\|\mb  V\mb  B\|$ and  $v_{\max}\|\mb  B\|\geq \|\mb  V\mb  B\|$, we use \cite[Proposition 7.2]{Am} to obtain
\begin{equation}
\lim\limits_{\e\to 0^+}\e(\mb  I - (\mb  I+\e\mb  B)\mb  E_{\e\la}(1))^{-1}  = (\la -\mb  V \mb  B)^{-1}\mb  V.
\label{limres}
\end{equation}
Next, using the first equation in (\ref{E}) we see that
$$
\|\!\!\cl{0}{1}\!\!\mb  E_{\e\la}(1-s)\mb f(s)ds - \cl{0}{1}\!\mb f(s)ds\| \leq \e\|\mb f\|\!\cl{0}{1}\! \|R_0(1-s)\|ds \leq \e\frac{\la\|\mb f\|}{v_{\max}} \cl{0}{1} \!(1-s)ds = \e\frac{\la\|\mb f\|}{2v_{\max}}.
$$
Thus
\begin{equation}
\lim\limits_{\e\to 0^+}\mb c_\e = (\la -\mb  V\mb  B)^{-1}\cl{0}{1}\mb f(s)ds.
\label{clim}
\end{equation}
Finally, treating $\mb c \to \mb  E_{\e\la}(x)\mb c$ as the operator from $l^1_\ms N$ to $\mb X$, we estimate
$$
\|\mb  E_{\e\la}(x)\mb c - \mb c\| = \cl{0}{1}\|(\mb  E_{\e\la}(x)-\mb  I)\mb c \|dx \leq \e\|\mb c\|\cl{0}{1}\|\mb  R_0(x)\| dx  \leq \e\frac{\la}{2v_{\max}}\|\mb c\|.
$$
This actually shows that the operators converge in the uniform operator norm. Combining all estimates and using the projection operator $\mb  P$ we find that
\begin{equation}
\lim\limits_{\e\to 0^+} R(\la,\mb A_{0,\e})= (\la -\mb  V\mb  B)^{-1}\mb  P.
\label{a0conv}
\end{equation}
From (\ref{abr1}) and (\ref{abr0}) we have, in particular,
\begin{equation}
\| R(\la, \mb A_{0,\e})\| \leq  \frac{\|R(\omega, \mb A_{0,\e} )\|}{c(\la-\omega)}\leq \frac{L}{\omega(\la-\omega)}.
\label{abr2}
\end{equation}
for some fixed  $\omega > v_{\max}\|\mb  B\|$ and arbitrary $\la>\omega$.  Since $\mb  Q=-\mb  M+\mb  R $ is bounded,
$$
\|(\mb  Q R(\la,\mb A_{0,\e}))^n\|\leq \frac{\|\mb  Q\|^nL^n}{\omega^n(\la-\omega)^n}.
$$
Hence, for $\mb A_{0,\e} =  \mb A_0+\mb  Q$ we have
\begin{equation}
R(\la, \mb A_\e) = R(\la, \mb A_{0,\e})\sum\limits_{n=1}^\infty (\mb  Q R(\la,\mb A_{0,\e})^n)
\label{resae}
\end{equation}
and the series converges uniformly in $\e$ for $\la > \omega +\|\mb  Q\|L\omega^{-1}.$
Since the operators $\mb  B, \mb  V, \mb  Q$ are independent of $x$, they commute with $\mb  P$ and, by $\mb  P^2=\mb  P$, we have
$$
\lim\limits_{\e\to 0^+}R(\la, \mb A_\e) = R(\la, \mb  V\mb  B)\mb  P\sum\limits_{n=1}^\infty (\mb  Q R(\la,\mb  V\mb  B)\mb  P)^n = R(\la, \mb  V\mb  B +\mb  Q)\mb  P.
$$
\end{proof}

\begin{corollary}
If $\mathring{\mb u}\in l^1_{\ms N}$ (that is, the initial condition is independent of $x$), then
\begin{equation}
\lim\limits_{\e\to 0^+} e^{t\mb A_\e} \mathring{\mb u} = e^{t(\mb  V\mb  B +\mb  Q)}\mathring{\mb u}
\label{semconv}
\end{equation}
almost uniformly (that is, uniformly on compact subsets) on $[0,+\infty)$.\label{regcon}
\end{corollary}
\begin{proof} According to the version of the Trotter--Kato approximation theorem given in \cite[Theorem 8.4.3]{Bobks}, if the resolvents of the generators of an equibounded family of semigroups (strongly) converge to an operator $R_\la$, then $R_\la$ is the resolvent of the generator of a semigroup on the closure of its range. Here, the limit operator is the resolvent of the bounded operator $\mb  V\mb  B+\mb Q$ composed with the projection $\mb  P:\mb X\to  l^1_{\ms N}$. Thus, restricted to $l^1_{\ms N}$, the range of $R(\la, \mb  V\mb  B +\mb  Q)\mb  P$ equals the range of $R(\la, \mb  V\mb  B +\mb  Q)$ which is $l^1_{\ms N}$. Hence, (\ref{semconv}) follows from the Trotter--Kato theorem. \end{proof}

\section{A counterexample}\label{sec2}This result is not very satisfactory. In asymptotic theory, \cite{BaLabook, Bobks1, BFN3}, the convergence obtained in Corollary \ref{regcon} is referred to as  regular convergence. However, typically it is possible to extend the convergence to initial data from the whole space, albeit at the cost of losing the convergence at $t=0,$ or adding necessary initial or boundary layers. The following example shows that this is impossible to achieve in our context.  Consider the scalar problem
\begin{eqnarray*}
\p_t u_\e(x,t) + \e^{-1}\p_x u_\e(x,t) &=&0,\quad x\in(0,1), t\geq 0,\nn\\
 u_\e(x,0) &=& \mathring{u}(x),\nn\\
 u_\e(0,t) &=& (1 +\e b) u_\e(1,t),
\end{eqnarray*}
where $b\in \mbb R$. It is easy to see that if $u_\e(x,t) = [e^{t\mb A_{0,\e}}\mathring u](x)$, then for $t$ satisfying $n\e \leq t<(n+1)\e,$ so that $n = \lfloor t/\e\rfloor$, we have
\begin{equation}
u_\e(x,t) = \left\{\begin{array}{lcl}(1+\e b)^{\lfloor\frac{t}{\e}\rfloor+1}\mathring u\left(x+\lfloor \frac{t}{\e}\rfloor+1-\frac{t}{\e}\right)&\mathrm{for}& 0\leq x \leq \frac{t}{\e} -\lfloor \frac{t}{\e}\rfloor,\\
(1+\e b)^{\lfloor \frac{t}{\e}\rfloor}\mathring u\left(x+\lfloor \frac{t}{\e}\rfloor-\frac{t}{\e}\right)&\mathrm{for}& \frac{t}{\e} -\lfloor \frac{t}{\e}\rfloor\leq x \leq 1.
\end{array}\right.
\label{usol}
\end{equation}
 Then
$$
\lim\limits_{\e\to 0^+} (1+\e b)^{n} = \lim\limits_{\e\to 0^+} \left((1+\e b)^{\frac{1}{\e b}}\right)^{\lfloor t/\e\rfloor \e b} = e^{bt},
$$
where we used
\begin{equation}
\lim\limits_{\e\to 0^+} \left\lfloor \frac{t}{\e}\right\rfloor{\e} =t.
\label{floorlim}
\end{equation}
The above is obvious for $t=0$ and for $t>0$ it follows from
$$
\frac{n}{n+1}\leq \left\lfloor \frac{t}{\e}\right\rfloor\frac{\e}{t}\leq 1
$$
and $n\to \infty$ with $\e\to 0$.

At the same time, let us consider $t=1$ and $\e = 1/k$. Then we obtain
$$
u_{\frac{1}{k}}(x,1) = \left(1+\frac{1}{k} b\right)^{k}\mathring u(x), \quad 0\leq x\leq 1,
$$
while for $\e= \frac{2}{2k+1}$ we have
$$
u_{\frac{2}{2k+1}}(x,1)= \left\{\begin{array}{lcl}(1+\frac{2}{2k+1} b)^{k+1}\mathring u\left(x+\frac{1}{2}\right)&\mathrm{for}& 0\leq x <\frac{1}{2},\\
\left(1+\frac{2}{2k+1}b\right)^{k}\mathring u\left(x-\frac{1}{2}\right)&\mathrm{for}& \frac{1}{2} \leq x \leq 1.
\end{array}\right.
$$
From this it follows that
$$
\lim\limits_{k \to \infty} e^{\mb A_{0,\frac{1}{k}}}\mathring u = e^b \mathring u
$$
and
$$
\lim\limits_{k \to \infty} e^{\mb A_{0,\frac{2}{2k+1}}}\mathring u = e^b \mathring v
$$
in $L_1([0,1]$, where $\mathring v(x) = \mathring u(x+1/2)$ for $x \in [0,1/2)$ and $\mathring v(x) = \mathring u(x-1/2)$ for $x \in [1/2,1
].$
Thus $\sem{\mb A_{0,\e}}$ does not converge as $\e\to 0^+$ for all initial conditions. However, it is easy to see that
$$
\lim\limits_{\e \to 0^+} e^{t\mb A_{0,\e}}\mathring u = e^{bt} \mathring u,
$$
provided $\mathring u$ is a constant, in accordance with Corollary \ref{regcon}.
\section{A version of irregular convergence}\label{sec3}
As demonstrated in Section \ref{sec2}, we should not expect the convergence of \sem{\mb A_{0,\e}} for initial conditions which are not constant for each population. However, under certain additional assumptions we can derive a stronger version of convergence than that of Corollary \ref{regcon},  which is of relevance to the problem at hand. In fact, we are interested in approximating the solutions to (\ref{rot1e}) (or (\ref{rot2e})) by solutions to the system of ODEs (\ref{B}). The solution of (\ref{B}) gives the total size of each subpopulation, while the solution to (\ref{rot1e}) (or (\ref{rot2e})) gives the density in each subpopulation. The total number of individuals in each subpopulation can be calculated by integrating the density. Hence, we can ask how well the total population in each subpopulation, calculated from the  solution of (\ref{rot1e}) ((or (\ref{rot2e})) can be approximated by the solution to (\ref{B}). Using the notation of Section \ref{s2}, the problem can be expressed as follows: does
\begin{equation}
\lim\limits_{\e\to 0^+} \mb  Pe^{t\mb A_\e}\mathring{\mb u } = e^{t\mb  H}\mb  P\mathring{\mb u}
\label{quest}
\end{equation}
hold for some some matrix $\mb  H$?

In this section we shall focus on problem (\ref{rot2e}) and adopt the assumption from \cite{KS04} that  the speeds $v_j$, $j\in \ms N,$ are linearly dependent over the field of rational numbers $\mbb Q$ or, in other words,
\begin{equation}
\exists_{v\in \mbb R}\forall_{j \in \ms N } \; \frac{v}{v_j} = l_j \in \mbb N.
\label{LD}
\end{equation}
In our interpretation of the model, this corresponds to the situation that the maturation times $\tau_j = 1/v_j$ in each subpopulation are natural multiples of some fixed reference maturation time $\tau =1/v$. We observe that (\ref{cmax}) implies that the set of different velocities is finite.

  Condition (\ref{LD}) allows the problem to be transformed into an analogous problem with unit velocities. Such a transformation appeared in \cite{KS04} (and in a more detailed version in \cite{Nathe}) in the context of transport on networks and thus, even though (\ref{rot2e}) is not necessarily related to the network transport, see \cite{BF1}, its interpretation as a network problem allows for a better description of the construction.

      Using the graph theoretical terminology, we identify the $j$th subpopulation with an oriented edge $e_j$ parametrized by $x\in (0,1)$, with ageing corresponding to moving from the tail of $e_j$, associated with 0, to the head associated with 1, and denote $G= \{e_j\}_{j\in \ms N}$. Let $Q$ be the graph with the set of vertices, $V(Q),$ equal to $G$ and the adjacency matrix $\mb I+\e\mb B$. We note that, in general,  $Q$ is not the line graph as $\mb I+\e\mb B$ allows for the construction of the vertices connecting the edges $e_j$  so that $G$ becomes a graph only in very special cases, \cite{BF1}.

  To proceed with the construction, first we re-scale time as $\tau = vt$ and, for each $j\in \ms N$,  we re-parameterize each edge $e_j$ by $y = l_jx.$ This converts (\ref{rot2e}) into a problem with the unit velocity for each $j\in \ms N,$ but with the equations defined on  $(0,l_j)$. However, since each $l_j$  is a natural number, we subdivide each interval $(0,l_j)$ into $l_j$ intervals of unit length. The $k$th subinterval in $(0,l_j)$ is identified with the edge $e_{j,k}$. This  creates from $G$ the set $G_1 =\{\{e_{j,k}\}_{1\leq k\leq l_j}\}_{j\in \ms N}$ consisting of, say, $\ms M$ edges. Then, as in the previous paragraph, we define $Q_1$ to be the graph with $V(Q_1) =G_1.$ Its adjacency matrix must take into account the connections between  the old edges $e_j,$ determined by $\mb  I +\e\mb  B,$  and the connections across the points subdividing the old edges. More precisely, consider a function $\mb f$ on $G$ that is continuous on each $e_j.$  This function is transformed into a function $\bof $ on $G_1,$ the components of which are required to have the same values at the head of    $e_{j,k}$ and the tail of $e_{j,k+1}$, $k=1,\ldots,l_j-1, j\in \ms N$. Then it is easy to see that the adjacency matrix of $Q_1$ satisfying these conditions is given by
  $$
   \mc T + \e \mc C = \left(\begin{array}{cccccc}\mc T_1&0&0&\ldots&0&\ldots\\
\vdots&\ddots&\vdots&\vdots&\vdots&\vdots\\
0&\ldots&\mc T_j&\ldots&0&\ldots\\
  \vdots&\vdots&\vdots&\ddots&\vdots&\ldots\end{array}\right) + \e\left(\begin{array}{cccc}\mc C_{11}&\ldots&\mc C_{1j}&\ldots\\
\vdots&\vdots&\ddots&\vdots\\
\mc C_{j1}&\ldots&\mc C_{jj}&\ldots\\
  \vdots&\vdots&\vdots&\vdots\end{array}\right).
  $$
  Here,  $\mc T_j$ is a $l_j\times l_j$ dimensional matrix which is either a scalar 1 if $l_j=1$, or a cyclic matrix, and $\mc C_{ij}$ is an $l_i\times l_j$ matrix, given, respectively, by
  $$
  \mc T_j = \left(\begin{array}{ccccc}0&0&\ldots&0&1\\
1&0&\ldots&0&0\\
\vdots&\vdots&\vdots&\vdots&0\\
0&0&\ldots&1&0\end{array}\right),
  \qquad
  \mc C_{ij} = \left(\begin{array}{ccccc}0&0&\ldots&0&b_{ij}\\
0&0&\ldots&0&0\\
\vdots&\vdots&\vdots&\vdots&0\\
0&0&\ldots&0&0\end{array}\right).
  $$
     Summarizing, we converted (\ref{rot2e}) into
   \begin{eqnarray}
\p_t \boldsymbol \upsilon_\e(x,t) + \e^{-1}\p_x\bu_\e(x,t) &=&0,\quad x\in(0,1), t\geq 0,\nn\\
\bu_\e(x,0) &=& \mathring{\bu}(x),\nn\\
\bu_\e(0,t) &=& (\mc T +\e \mc C)\bu_\e(1,t).
\label{rot2e1}
\end{eqnarray}
   Since (\ref{rot2e1})  has the same structure as (\ref{rot2e}),  there is a semigroup  $\sem{\mc A_{0,\e}}$ solving it.

To be more precise, the above construction defines an operator $\mb f \to \bof = \mbb S\mb f$, where
$\bof= (\phi_{j,s})_{j\in \ms N, 1\leq s\leq l_j},$ $\mb f = (f_j)_{j\in \ms N}$ and, for $s\in \{1,\ldots, l_j\},$
\begin{equation}
\phi_{j,s} (y)= f_j|_{\left[\frac{s-1}{l_j},\frac{s}{l_j}\right)}\left(\frac{s+y-1}{l_j}\right).
\label{Uu}
\end{equation}
It is easy to see that the inverse $\mb f =\mbb S^{-1}\bof$ is defined by
\begin{equation}
f_j(x) = \phi_{j,s}(l_jx +1-s), \qquad x \in \left[\frac{s-1}{l_j}, \frac{s}{l_j}\right), \;s \in\{1,\ldots,l_j\},\;j \in \ms N.
\label{uU}
\end{equation}
By direct calculation, see also \cite{Nathe},  $\mbb S: \mb X \to \mc X := L_1([0,1],l^1_\ms M)$ is an isomorphism such that we have the similarity relation
\begin{equation}
e^{t\mb A_{0,\e}} \mb f = \mbb S^{-1} e^{tv\mc A_{0,\e}}  \mbb S \mb f, \qquad \mb f \in \mb X.
\label{sim}
\end{equation}
The motivation behind (\ref{sim}), see \cite{BD08} and \cite[Proposition 4.5.1]{Nathe}, is the fact that for any $\bof \in \mc X$
\begin{equation}
(e^{tv \mc A_{0,\e}}\bof)(x) = (\mc T+\e\mc C)^n\mc  \bof\left(n+x-\frac{vt}{\e}\right), \quad n\in \mbb N, \quad 0\leq n+x-\frac{vt}{\e}<1,
\label{TB}
\end{equation}
with $(e^{tv \mc A_{0,\e}}\bof)(x) = \bof(x-vt/\e)$ for $vt/\e\leq x< 1$.

Let $l = lcm\{l_j\}_{j\in \ms N}$. We observe that $\mc T^l = \mc I$ and
\begin{equation}
(\mc T +\e\mc C)^l = \mc I+ \e \ti{\mc C} +\e^2\mc D,
\label{kexp}
\end{equation}
where
$$
\ti {\mc C} = \sum\limits_{i=0}^l \mc T^{l-1-i}\mc C\mc T^{i}.
$$
For any $\ms K\subset \mbb N$, let $\mb P_\ms K$ be the projection from $L_1([0,1], l^1_\ms K)$ to $l^1_\ms K,$  given by (\ref{mbP}).

The next result is not strictly necessary but it relates operators on $G$ to those on $G_1$ and  introduces in a natural way the projection $\Pi$ that plays an essential role in the main theorem.
\begin{proposition}
For $\la>v_{\max}\|\mb B\|$, we have 
\begin{equation}
\left(\la v -l^{-1}\ti {\mc C}\right)^{-1}\Pi \mb P_\ms M\bof = \mbb S(\la - \mb V\mb B)^{-1}\mb P_\ms N\mbb S^{-1}\bof,
\label{reseq1}
\end{equation}
where \begin{equation}
\Pi= l^{-1}\sum\limits_{i=0}^{l-1} \mc T^i.
\label{Pi}
\end{equation}
\end{proposition}
\begin{proof}
The estimate (\ref{neuest}) carries over to this case by (\ref{sim}), which also  gives
\begin{equation}
R(\la v, \mc A_{0,\e}) \bof = \mbb S R(\la,\mb A_{0,\e})  \mbb S^{-1} \bof, \qquad \bof \in \mc X, \la>v_{\max}\|\mb B\|,
\label{sim1}
\end{equation}
and, by Theorem \ref{rescon} and the continuity of $\mbb S$, 
\begin{equation}
\lim\limits_{\e\to 0^+}R(\la, \mc A_{0,\e}) \bof  = \mbb S(\la - \mb V\mb B)^{-1}\mb P_\ms N\mbb S^{-1}\bof.
\label{ressim}
\end{equation}
To find the limit resolvent in terms of $\ti {\mc C}$, we modify the calculations from Theorem \ref{rescon}. The only different part is related to the calculation of \begin{equation}
\mb c_\e = \e(\mc I - (\mc T+\e\mc C)\mc E_{\e\la}(1))^{-1}(\mc T+\e\mc C) \cl{0}{1}\mc E_{\e\la}(1-s)\mb f(s)ds,
\label{resk1'}
\end{equation}
where here $\mc E_{\e\la}(s) = e^{-\e\la s}\mc I$.
We have, for $\la>v_{\max}\|\mc B\|$,
\begin{eqnarray*}
(\mc I - (\mc T+\e\mc C)\mc E_{\e\la}(1))^{-1} &=& \sum\limits_{k=0}^\infty e^{-\e\la k}(\mc T+\e\mc C)^k \\
&=&\sum\limits_{i=0}^{l-1} (\mc T+\e\mc C)^i\left( \sum\limits_{j=0}^\infty (\mc T+\e\mc C)^{lj}e^{-\e\la (lj+i)}\right)\\
&=& \sum\limits_{j=0}^\infty (\mc I+\e\ti{\mc C} +\e^2\mc D)^j\mc E_{\e l\la}(j)\left( \sum\limits_{i=0}^{l-1} (\mc T+\e\mc C)^i e^{\e\la i}\right).
\end{eqnarray*}
Hence, as in (\ref{limres}) with $\mb V = l^{-1}\mc I$, 
$$
\lim\limits_{\e\to 0^+} \e(\mc I - (\mc T+\e\mc C)\mc E_{\e\la}(1))^{-1} = \left(\la -l^{-1}\ti {\mc C}\right)^{-1}\left(l^{-1}\sum\limits_{i=0}^{l-1} \mc T^i\right)=\left(\la -l^{-1}\ti {\mc C}\right)^{-1}\Pi.
$$
 Since $\mc T$ is block diagonal with a finite number of finite dimensional blocks $\mc T_j$ (different from  $1$), we can use the finite dimensional theory to claim, by \cite[p. 633]{CD00}, that $\Pi$ is the projector onto the eigenspace of $\mc T$ corresponding to the eigenvalue $\la =1$ along the range of $\mc I-\mc T$. The action of each block
$$
\Pi_j = l^{-1}\sum\limits_{j=0}^{l-1} \mc T_j^i = l_j^{-1}\sum\limits_{j=0}^{l_j-1} \mc T_j^i
$$
amounts to
\begin{equation}
\Pi_j\bu_j = \frac{1}{l_j}\left(\sum\limits_{r=1}^{l_j}\upsilon_{j,r},\ldots,\sum\limits_{r=1}^{l_j}\upsilon_{j,r}
\right).
\label{PiS}
\end{equation}
 Then the operator  $\mbb S^{-1}\Pi$ transforms functions which are constant on each edge of $G_1$ to functions which are constant on each edge of $G$.

Now, proceeding as in the proof of Theorem \ref{rescon}, we find
$$
\lim\limits_{\e\to 0^+} R(\la, \mc A_{0,\e})\bof = \left(\la -l^{-1}\ti {\mc C}\right)^{-1}\left(l^{-1}\sum\limits_{i=0}^{l-1} \mc T^i\right)\cl{0}{1}\bof(s)ds = \left(\la -l^{-1}\ti {\mc C}\right)^{-1}\Pi \mb P_\ms M\bof,
$$
where we used \begin{equation}\left(\sum\limits_{i=0}^{l-1} \mc T^i\right)\mc T = \sum\limits_{i=0}^{l-1} \mc T^i,\label{per}\end{equation}
by periodicity. This, combined with (\ref{sim1}), ends the proof. \end{proof}
Using (\ref{PiS}) and (\ref{uU}), we see that
\begin{equation}
\mb P_\ms N\mbb S^{-1}\bof = \cl{0}{1}[\mbb S^{-1}\bof](x)dx = \left(l_j^{-1} \sum\limits_{s=1}^{l_j}\cl{0}{1} \phi_{j,s}(y)dy\right)_{j\in \ms N} = \mbb S^{-1}\Pi\mb P_{\ms M}\bof.\label{bof'}
\end{equation}
\begin{theorem}
For any $\mb   f \in \mb X$ we have
\begin{equation}
\lim\limits_{\e\to 0^+} \mb P_{\ms N}e^{t \mb A_{0,\e}}\mb f = e^{t\mb V\mb B}\mb P_\ms N\mb f.
\label{main}
\end{equation}
\end{theorem}
\begin{proof}
 Let us write $\mb f \in \mb X$ as
$$
\mb f = \mb P_{\ms N}\mb f + \mb f - \mb P_{\ms N}\mb f = \mb P_{\ms N}\mb f +\mb w,
$$
where $\mb P_{\ms N}\mb f\in l^1_{\ms N}$ and $\mb P_{\ms N}\mb w =0$. 

By Corollary \ref{regcon} and the continuity of $\mb P_\ms N,$
$$
\lim\limits_{\e\to 0^+}\mb P_\ms N [e^{t\mb A_{0,\e}}\mb P_\ms N \mb f] = \mb P_\ms N e^{t \mb V\mb B}\mb P_\ms N\mb f = e^{t \mb V\mb B}\mb P_\ms N\mb f.
$$
Hence, by linearity, it suffices to show that 
\begin{equation}
\lim\limits_{\e\to 0^+}\mb P_\ms N [e^{t\mb A_{0,\e}}\mb w] = 0,
\label{w0}
\end{equation}
provided $\mb P_\ms N\mb w =0$. By (\ref{bof'}), we have
\begin{equation}
\mb P_\ms N [e^{t\mb A_{0,\e}}\mb f] = \mb P_\ms N \mbb S^{-1} [e^{tv\mc A_{0,\e}}\mbb S\mb f] =
\mbb S^{-1}\Pi\mb P_\ms M  [e^{tv\mc A_{0,\e}}\mbb S\mb f], \quad \mb f \in \mb X.
\label{go}
\end{equation}
Further, for $n-1 \leq vt/\e\leq n$ and $\bof \in \mc X,$ we have
\begin{eqnarray}
&&\mb P_\ms M e^{tv \mc A_{0,\epsilon} }\bof\nn\\
&&\phantom{}=(\mc{T}+\epsilon \mc{C})^{n}\!\!\!\!\!\!\!\cl{0}{\frac{vt}{\epsilon}-n+1}\!\!\!\!\!\bof\left(n+x-\frac{vt}{\epsilon}\right)dx+
(\mc{T}+\epsilon \mc{C})^{n-1}\!\!\!\!\!\!\!\cl{\frac{vt}{\epsilon}-n+1}{1}\!\!\!\!\!\bof\left(n-1+x-\frac{vt}{\epsilon}\right)dx \nn\\
&&\phantom{}=\mc T(\mc{T}+\epsilon \mc{C})^{n-1}\!\!\cl{0}{1}\!\!\bof(z)dz+\e\mc C(\mc{T}+\epsilon \mc{C})^{n-1}\!\!\!\cl{n-\frac{vt}{\epsilon}}{1}\!\!\!\bof(z)dz.
\label{mcP}
\end{eqnarray}
Let us denote $(\mc{T}+\epsilon \mc{C})^{j} -\mc T^j =\e\mc R_j.$ Then, by (\ref{per}),
\begin{eqnarray}
\Pi \mc T(\mc{T}+\epsilon \mc{C})^{n-1} &=& \Pi(\mc{T}+\epsilon \mc{C})^{n-1}=\frac{1}{l} \sum\limits_{j=0}^{l-1}\mc T^j(\mc{T}+\epsilon \mc{C})^{n-1}  \label{hm}\\
&=& \frac{1}{l} \sum\limits_{j=0}^{l-1}\left((\mc{T}+\epsilon \mc{C})^{n-1+j} -\e\mc R_j(\mc{T}+\epsilon \mc{C})^{n-1}\right)\nn\\
&=&\frac{1}{l}\sum\limits_{j=0}^{l-1}\left((\mc{T}+\epsilon \mc{C})^{n-1}(\mc T^j +\e\mc R_j) -\e\mc R_j(\mc{T}+\epsilon \mc{C})^{n-1}\right) \nn\\
&=& (\mc{T}+\epsilon \mc{C})^{n-1}\Pi  + \e\left((\mc{T}+\epsilon \mc{C})^{n-1}\mc R -\mc R(\mc{T}+\epsilon \mc{C})^{n-1}\right),\nn
\end{eqnarray}
where $\mc R = \frac{1}{l}\sum\limits_{j=0}^{l-1}\mc R_j$. Since $\mbb S$ is an isomorphism, using (\ref{bof'}) we see that $\mb P_\ms N\mb w=0$ implies $\Pi\mb P_\ms M\mbb S\mb w =0$, hence $(\mc{T}+\epsilon \mc{C})^{n-1}\Pi \mb P_\ms M \mbb S\mb w=0$. Thus (\ref{mcP}) and (\ref{hm}) yield
\begin{eqnarray*}
\lim\limits_{\e\to 0}\Pi\mb P_\ms M e^{tv \mc A_{0,\epsilon} }\mbb S\mb w &=&\lim\limits_{\e\to 0}\e\Pi \mc C(\mc{T}+\epsilon \mc{C})^{n-1}\!\!\!\cl{n-\frac{vt}{\epsilon}}{1}\!\!\!\mbb S\mb w(z)dz\nn\\
&&\phantom{xx} + \lim\limits_{\e\to 0}\e\left((\mc{T}+\epsilon \mc{C})^{n-1}\mc R -\mc R(\mc{T}+\epsilon \mc{C})^{n-1}\right)\mb P_\ms M\mbb S\mb w =0,
\label{mcP1}
\end{eqnarray*}
which proves (\ref{w0}).
\end{proof}

\medskip
Received ; Accepted .

\medskip


\begin{thebibliography}{[KLR73]}
  \bibitem{Am} \newblock H. Amman, J. Escher, \newblock "Analysis II", \newblock Birkh\"{a}user, Basel (2008).
  \bibitem{Ar} \newblock W. Arendt, \newblock \textit{Resolvent positive operators}, \newblock {Proc. Lond.
Math. Soc.}  (3) 54 (1987), 321--349.
\bibitem{BaAr} \newblock J. Banasiak, L. Arlotti, \newblock {"Positive perturbations of semigroups with applications"}, \newblock Springer Verlag, London, 2006.
\bibitem{BaLabook} \newblock J. Banasiak, M. Lachowicz, \newblock{"Methods of Small Parameter in Mathematical Biology"}, \newblock Birkh\"auser/Springer, Cham, 2014.
\bibitem{BM} \newblock J. Banasiak, M. Moszy\'{n}ski, \newblock\textit{Dynamics of birth-and-death processes with proliferation – stability and chaos,} \newblock Discrete Contin. Dyn. Syst. \textbf{29}(1), (2011)
67--79.
\bibitem{BF1} \newblock J. Banasiak, A. Falkiewicz, \newblock \textit{Some transport and diffusion processes on networks and their graph realizability,} \newblock {Appl. Math. Lett.}, \textbf{45}, (2015), 25--30
\bibitem{BFN2} \newblock J. Banasiak, A. Falkiewicz, P. Namayanja, \newblock\textit{Semigroup approach to diffusion and transport problems on networks}, \newblock {Semigroup Forum}, DOI 10.1007/s00233-015-9730-4.
    \bibitem{BFN3} \newblock J. Banasiak, A. Falkiewicz, P. Namayanja, \newblock\textit{ Asymptotic state lumping in transport and diffusion problems on networks  with applications to population problems}, \newblock Math. Models Methods Appl. Sci., \textbf{26}(2),  (2016),  215--247, DOI: 10.1142/S0218202516400017.
  \bibitem{Bobks1} \newblock A. Bobrowski,  \newblock  "Convergence of One-parameter Operator Semigroups. In Models of Mathematical Biology and Elsewhere", \newblock Cambridge University Press, Cambridge, 2016.
\bibitem{Bobks} \newblock A. Bobrowski,  \newblock {"Functional Analysis for Probability and Stochastic Processes"}, \newblock Cambridge University Press, Cambridge, 2005.
\bibitem{BoK} \newblock A. Bobrowski, M. Kimmel, \newblock\textit{Asymptotic behaviour of an operator exponential related to branching random walk models of DNA repeats}, \newblock J. Biol. Systems, \textbf{7}(1), (1999), 33--43.
    \bibitem{BD08}\newblock
 B. Dorn, \newblock\textit{Semigroups for flows in infinite networks}, \newblock {Semigroup Forum}, {\bf 76},
(2008), 341--356.
\bibitem{DS} \newblock N. Dunford, J. T. Schwartz, \newblock "Linear Operators. Part I: General Theory", \newblock Wiley-Interscience, New York, 1988.
\bibitem{EN} \newblock K.-J. Engel, R. Nagel, \newblock{"One-Parameter Semigroups
for Linear Evolution Equations"}, Springer Verlag, New York, 1999.
\bibitem{KS} \newblock M. Kimmel, D.N. Stivers, \newblock \textit{ Time-continuous branching
walk models of unstable gene amplification}, \newblock { Bull. Math.
Biol.} {\bf 50}, (1994), 337--357.
\bibitem{AD+P1}\newblock
M. Kimmel, A. {\'S}wierniak and A. Pola{\'n}ski,\newblock\textit{
Infinite--dimensional model of evolution of drug resistance of
cancer cells}, \newblock{ J. Math. Systems Estimation Control} {\bf
8}(1), 1998, 1--16.
\bibitem{KS04} \newblock M. Kramar, E. Sikolya, \newblock\textit{Spectral Properties and Asymptotic Periodicity of Flows in Networks},
\newblock {Math. Z.}, {\bf 249}, (2005), 139--162.
\bibitem{leb} \newblock J. L. Lebowitz, S. I. Rubinov, \newblock\textit{A Theory for the Age and Generation Time Distribution of a Microbial Population}, {J. Theor. Biol.}, \textbf{1}, (1974), 17--36.
    \bibitem{CD00} \newblock  C. D. Meyer, \newblock{"Matrix Analysis and Applied Linear Algebra"},  \newblock SIAM, Philadelphia, 2000.
    \bibitem{Nathe}
 P. Namayanja,  {"Transport on Network Structures"}, Ph.D thesis, UKZN, 2012.
\bibitem{rot}
\newblock M. Rotenberg, \newblock\textit{Transport theory for growing cell population,} {J. Theor. Biol.}, \textbf{103}, (1983), 181--199.
\bibitem{Ru} W. Rudin, \newblock "Functional analysis", \newblock McGraw-Hill Book Co., New York, 1973.
 \bibitem{SPK} \newblock A. \'{S}wierniak, A. Pola\'{n}ski,  M. Kimmel, \newblock\textit{Control problems arising in chemotherapy
under evolving drug resisitance}, \newblock Preprints of the 13th World Congress
of IFAC 1996, Volume B, 411--416.
\bibitem{TWC} \newblock H. T. K. Tse, W. McConnell Weaver,  D. Di Carlo, \newblock\textit{Increased Asymmetric and Multi-Daughter Cell Division in Mechanically Confined Microenvironments}, PLoS ONE, (2012), 7(6): e38986,  doi:10.1371/journal.pone.0038986.












\end{thebibliography}
\end{document}